\documentclass[12pt,reqno]{amsart}
\usepackage{amsmath, amsthm, amssymb}
\usepackage{mathtools}
\usepackage{mathrsfs}  
\usepackage{stmaryrd}
\usepackage{comment}
\usepackage{marginnote}

\usepackage[margin=1in]{geometry}
\usepackage{parskip}
\usepackage[shortlabels]{enumitem}

\usepackage{color}

\usepackage[bookmarks,colorlinks,breaklinks]{hyperref}
\hypersetup{linkcolor=red,citecolor=blue,
  filecolor=dullmagenta,urlcolor=darkblue}

\usepackage[numbers]{natbib}
\usepackage{cleveref}

\usepackage{graphicx}
\usepackage{tikz}
\usetikzlibrary{matrix, arrows}
\usetikzlibrary{cd}

\usepackage{hyperref}
\usepackage{enumitem}
\usepackage{graphicx}

\makeatletter
\def\thm@space@setup{%
  \thm@preskip=0.5em\thm@postskip=\thm@preskip%
}
\makeatother

\newtheoremstyle{named}{}{}{\\itshape}{}{\bfseries}{.}{.5em}{\thmnote{#3's }#1}
\theoremstyle{named}

\theoremstyle{plain}
\newtheorem{thm}{Theorem}[section]

\newtheorem{prop}[thm]{Proposition}
\newtheorem{lem}[thm]{Lemma}
\newtheorem{cor}[thm]{Corollary}
\theoremstyle{definition}
\newtheorem{defn}[thm]{Definition}

\theoremstyle{remark}
\newtheorem{rmk}[thm]{Remark}
\crefformat{footnote}{#2\footnotemark[#1]#3}

\newcommand{\Hom}{\mathrm{Hom}}

\newcommand{\fg}{\mathfrak{g}}

\newcommand{\A}{\mathbb{A}}
\newcommand{\CC}{\mathbb{C}}
\newcommand{\QQ}{\mathbb{Q}}

\newcommand{\ZZ}{\mathbb{Z}}

\newcommand{\GL}{\mathrm{GL}}

\newcommand{\Sp}{\mathrm{Sp}}

\newcommand{\cO}{\mathcal{O}}

\newcommand{\mf}{\mathfrak}
\newcommand{\mc}[1]{\mathcal{#1}}

\newcommand{\mr}[1]{\mathrm{#1}}
\newcommand{\ol}[1]{\overline{#1}}
\newcommand{\ul}[1]{\underline{#1}}

\newcommand{\Ad}{\mathrm{Ad}}
\newcommand{\ad}{\mathrm{ad}}

\newcommand{\et}{\text{\'et}}
\newcommand{\Spec}{\mathrm{Spec}}
\newcommand{\Loc}{\mathrm{Loc}}
\newcommand{\Rep}{\mathrm{Rep}}
\renewcommand{\top}{\mathrm{top}}

\DeclareMathOperator{\Pross}{Pro-ss}
\DeclareMathOperator{\im}{im}
\DeclareMathOperator{\LS}{LS}
\newcommand{\tLS}{\widetilde{\LS}}
\newcommand{\git}{{\,\!\sslash\!\,}}

\usepackage[utf8]{inputenc}

\newcommand{\spc}[1]{{\color{blue} \textsf{$\blacktriangle\blacktriangle\blacktriangle$ Comment: [#1]}}}

\title{G-rigid local systems are integral}

\author[C.~Klevdal]{Christian Klevdal}
\address{Department of Mathematics, The University of Utah, 155 S 1400 E, Salt Lake City, UT 84112, USA}
\email{klevdal@math.utah.edu}
\author[S.~Patrikis]{Stefan Patrikis}
\address{Department of Mathematics, The Ohio State University, 100 Math Tower, 231 West 18th Avenue, Columbus, OH 43210, USA}
\email{patrikis.1@osu.edu}
\date{September 2020}

\thanks{We thank Domingo Toledo for helpful conversations about lattice rigidity theorems, and we thank H\'{e}l\`{e}ne Esnault and Michael Groechenig for their comments. C.K. was supported by NSF grant DMS-1840190. S.P. was supported by NSF grants DMS-1700759 and DMS-1752313.}

\begin{document}

\begin{abstract}
    Let $G$ be a reductive group, and let $X$ be a smooth quasi-projective complex variety. We prove that any $G$-irreducible, $G$-cohomologically rigid local system on $X$ with finite order abelianization and quasi-unipotent local monodromies is integral. This generalizes work of Esnault and Groechenig when $G= \mathrm{GL}_n$, and it answers positively a conjecture of Simpson for $G$-cohomologically rigid local systems. Along the way we show that the connected component of the Zariski-closure of the monodromy group of any such local system is semisimple.
\end{abstract}

\maketitle

\section{Introduction}
A central question of arithmetic geometry is to identify which Galois representations arise, via \'{e}tale cohomology, from algebraic geometry. The Fontaine-Mazur conjecture (\cite{fontaine-mazur}), asserting that all potentially semistable and almost everywhere unramified semisimple $\ell$-adic representations of the absolute Galois group of a number field $F$ do indeed appear in the cohomology of smooth projective varieties over $F$, is the prototypical and most famous explicit problem in this area. Carlos Simpson formulated in \cite[pg. 9 Conjecture]{simpson:higgs} an analogous conjecture that \textit{rigid} complex representations of the topological fundamental group of a smooth projective complex variety $X$ necessarily appear in the cohomology of a family of varieties over $X$. In particular (\textit{loc. cit.}), rigid representations should be defined over the ring of integers in some number field. Simpson raised these questions for representations valued not just in $\mr{GL}_n$, but in general algebraic groups $G$.

When $G= \mr{GL}_n$, two striking general results are known. For $X$ an open subvariety of $\mathbb{P}^1$, Katz has proven that any $\mr{GL}_n$-cohomologically rigid (see Definition \ref{basicdefs}) irreducible local system on $X$ with quasi-unipotent local monodromies is a subquotient of the monodromy representation of a family of varieties over $X$ (\cite[Theorem 8.4]{katz:rls}). Esnault and Groechenig (\cite[Theorem 1.1]{esnault-groechenig:rigid}; see also \cite{esnault-groechenig:rigidconnection}) have proven Simpson's integrality  conjecture, for all smooth quasi-projective varieties, for $\mr{GL}_n$-cohomologically rigid irreducible representations with finite-order determinant and quasi-unipotent local monodromies. 

In the present paper we generalize the main theorem of \cite{esnault-groechenig:rigid} to rigid representations valued in general connected reductive groups $G$. We begin by making precise the basic terms.
\begin{defn}\label{basicdefs}
Let $X$ be a connected smooth quasi-projective variety over $\CC$, and let $j \colon X \hookrightarrow \ol{X}$ be a good compactification: thus $\ol{X}$ is smooth projective, and the boundary $D = \ol{X} \setminus X$ is a strict normal crossings divisor. Let $x \in X(\CC)$ be any base-point, and let $\pi_1^{\mr{top}}(X, x)$ be the corresponding topological fundamental group. A homomorphism $\rho \colon \pi_1^{\mr{top}}(X, x) \to G(\CC)$ 
\begin{itemize}
\item is $G$-irreducible if the image of $\rho$ is not contained in any proper parabolic subgroup of $G$.\footnote{Replacing $G(\CC)$ in this definition by $G(K)$ for some field $K$, we will tend to abuse notation and say a representation is $G$-irreducible if it is $G$-absolutely irreducible, i.e. the resulting homomorphism into $G(\ol{K})$ does not factor through a proper parabolic subgroup.}
\item is $G$-cohomologically rigid if $H^1(\ol{X}, j_{!*}\ul{\fg^{\mr{der}}})=0$, where $\fg^{\mr{der}}$ is the Lie algebra of the derived group of $G$, regarded as a local system on $X$ via the composite $\Ad \circ \rho$.
\item has quasi-unipotent local monodromy if for all points $y$ in the smooth locus of $D$ and any sufficiently small ball $\Delta \subset \ol{X}$ around $y$, $\rho(\gamma)$ is quasi-unipotent for a generator $\gamma$ of $\pi_1^\top(\Delta\setminus D \cap \Delta) \cong \ZZ$. 
\end{itemize}
\end{defn}
Equivalently, $\rho$ is cohomologically rigid if it represents a smooth isolated point on an appropriate moduli space of $G$-local systems; see Proposition \ref{tangent space} and Remark \ref{intermediate extension}. Note that if $\rho$ is either $G$-irreducible or $G$-cohomologically rigid, there need not exist a faithful finite-dimensional representation $r \colon G \to \mr{GL}_n$ of $G$ such that $r \circ \rho$ is either $\mr{GL}_n$-irreducible or $\mr{GL}_n$-cohomologically rigid; thus the results of \cite{esnault-groechenig:rigid} cannot be used to bootstrap to the case of general $G$.


We now state the main theorem:
\begin{thm}\label{mainthm}
Let $X/\CC$ be a connected smooth quasi-projective complex algebraic variety, with a base-point $x \in X(\CC)$. Let $G$ be a split connected reductive group over $\ZZ$, and let 
\[
\rho \colon \pi_1^{\mr{top}}(X, x) \to G(\CC)
\]
be a $G$-irreducible and $G$-cohomologically rigid local system such that
\begin{itemize}
    \item $\rho$ has quasi-unipotent local monodromy;
    \item the image of the composite homomorphism
    \[
    \pi_1^{\mr{top}}(X, x) \xrightarrow{\rho} G(\CC) \to A(\CC)
    \]
    to the maximal abelian quotient $A$ of $G$ has finite order.
\end{itemize}
Then $\ol{\mr{im}(\rho)}^0$, the connected component of the identity of the Zariski closure of the image $\mr{im}(\rho)$ of $\rho$, is semisimple; and there is a number field $L$ with ring of integers $\mc{O}_L$ such that $\rho$ is $G(\CC)$-conjugate to a homomorphism $\pi_1^{\mr{top}}(X, x) \to G(\mc{O}_L)$.
\end{thm}
\begin{rmk}
When $G= \mr{GL}_n$, and discounting the conclusion that $\ol{\im(\rho)}^0$ is semisimple, this is precisely \cite[Theorem 1.1]{esnault-groechenig:rigid}.
If $\rho$ comes from geometry, i.e.\ it arises as the monodromy representation of a sub-local system of $R^if_\ast \ul{\CC}$ for a smooth map $f \colon Y \to X$, then $\rho$ satisfies the two conditions of the theorem. Indeed, it first follows from the local monodromy theorem that $\rho$ has quasi-unipotent local monodromy. That the abelianized monodromy representation has finite image follows from \cite[Corollaire 4.2.8.iii(b)]{deligne:ht2}.


\end{rmk}
\begin{rmk}
Many naturally-occurring local systems are rigid and provably integral by other means. For instance, let $H$ be a connected semisimple Lie group with (real) rank at least 2 and having no compact factors, and let $\Gamma \subset H$ be an irreducible \textit{lattice} such that $\Gamma \cdot H^{\mr{is}}$ is dense in $H$, with $H^{\mr{is}}$ the minimal connected normal subgroup such that $H/H^{\mr{is}}$ is compact. Then Margulis has proven in turn the following remarkable results (see \cite[Theorems IX.6.5, IX.6.15]{margulis:discrete}, starting from his lattice superrigidity theorem:
\begin{itemize}
    \item $\Gamma$ is an arithmetic subgroup.
    \item For any homomorphism $\rho \colon \Gamma \to G(\CC)$, the Zariski-closure $\ol{\im(\rho)}$ is semisimple.
    \item For every representation $r \colon \Gamma \to \mr{GL}_n(\CC)$, $H^1(\Gamma, r)$ is trivial.
\end{itemize}
In particular, with $r$ equal to the adjoint action of $G$ on $\fg^{\mr{der}}$, if the associated locally symmetric space $\Gamma \backslash H$ has the structure of a complex quasi-projective variety, we see that $\rho$ is cohomologically rigid in the sense of Definition \ref{basicdefs}. (See Propositions \ref{tangent space group cohomology} and \ref{tangent space}; the vanishing condition here is in general stronger than what is needed for cohomological rigidity.) Many interesting rigid representations, however, have monodromy groups that are not lattices in their real Zariski-closures: this is the phenomenon of so-called \textit{thin} monodromy groups, and famous (hypergeometric) examples have been studied in \cite{deligne-mostow}, and more recently \cite{brav-thomas:thin}. In our algebro-geometric setting, the chain of reasoning is reversed: we assume cohomological rigidity, and then deduce the semisimplicity of the monodromy group and integrality (in place of arithmeticity) of the representation.
\end{rmk}

\subsection{Overview of the proof}

The proof follows the arguments of \cite{esnault-groechenig:rigid}, and indeed our debt to that paper will be evident throughout. The essential idea is, having shown the rigid representation is defined over the ring of $\Sigma$-integers $\mc{O}_{K, \Sigma} \subset K$ for some number field $K$ and finite set of places $\Sigma$, to check integrality at each $\lambda \in \Sigma$ by specializing $\rho$ to characteristic $p$ and using results of Drinfeld and Lafforgue on the existence of compatible systems of $\lambda$-adic representations. Such arguments are considerably subtler for general $G$ than for $\mr{GL}_n$, since in general the semisimple conjugacy classes associated to Frobenius elements do not uniquely characterize $G$-irreducible $\lambda$-adic representations. In particular, our argument must keep track of monodromy groups in a way that \cite{esnault-groechenig:rigid} does not, and we rely on Drinfeld's work \cite{drinfeld:pross} for the existence of the requisite compatible systems. 

Here is a more detailed section-by-section outline of the proof, restricting for notational simplicity to the case where $X$ is projective. In \S 3, we prove the local integrality condition needed for the main theorem: granted that $\rho$ factors as $\pi_1^{\mr{top}}(X, x) \xrightarrow{\rho} G(\mc{O}_{K, \Sigma})$, this reduces us to checking that for each $\lambda \in \Sigma$, the composite
 \[ \rho_\lambda \colon \pi_1^\top(X,x) \xrightarrow{\rho} G(\mc{O}_{K,\Sigma}) \to G(\ol{K}_\lambda) \]
can be conjugated into $G(\mc{O}_{\ol{K}_\lambda})$. 

The initial factorization of $\rho$ through some $G(\mc{O}_{K, \Sigma})$ is obtained, following \cite{simpson:higgs} and \cite{esnault-groechenig:rigid}, by studying a suitable moduli space of $G$-local systems on $X$. In \S 4, we recall the construction and basic properties of these moduli spaces. Our proof begins in earnest in \S 5. We consider a set $\mc{S}$ of (isomorphism classes of) $G$-local systems satisfying the hypotheses of the theorem. Using that the moduli space of $G$-local systems is finite type, it is deduced (here is the key input from cohomological rigidity) that $\mc{S}$ is finite, and that there exists a number field $K$ and a finite set of places $\Sigma$ such that each element of $\mc{S}$ is conjugate to a representation $\rho \colon \pi_1^\top(X,x) \to G(\mc{O}_{K,\Sigma})$. 
For each $\rho' \in \mc{S}$ and each place $\lambda$ of $K$, we obtain, via extending scalars, a representation $\rho'_{\lambda} \colon \pi_1^{\mr{top}}(X, x) \to G(\ol{K}_{\lambda})$, and we denote the collection of these homomorphisms by $\mc{S}_{\lambda}$. The final step of the proof (and the most interesting) is to deduce the integrality of our original $\rho_{\lambda}$, for each $\lambda \in \Sigma$, from the integrality of the members of $\mc{S}_{\lambda'}$ for a fixed $\lambda' \not \in \Sigma$ (note that in \S \ref{descentssection} and \S \ref{companionsection} the notational roles of $\lambda$ and $\lambda'$ are reversed).

We indicate this last step in more detail. The crucial inputs are results of L.\ Lafforgue (\cite{llafforgue:chtoucas}) on the Langlands correspondence over function fields, and results of Drinfeld (\cite{drinfeld:pross}, building on \cite{drinfeld:deligneconj}) that promote Lafforgue's work to construct $\lambda$-adic companions for $\lambda'$-adic representations---and even for suitable $G(\ol{K}_{\lambda'})$-representations---of the fundamental group of a smooth variety of any dimension over a finite field. To exploit these results, following \cite{esnault-groechenig:rigid}, we spread the complex variety $X$ out and and take a fiber $X_{s}$ over a finite field. In the remainder of \S 5 we show that the collection $\mc{S}_{\lambda'}$
can be specialized and descended to (\'{e}tale)
$G$-local systems $\{\rho'_{\lambda',s}\}_{\rho' \in \mc{S}}$ on $X_s$. This step is subtler than in the case of $\mr{GL}_n$; it requires attending to the monodromy groups of the specialized representations (Proposition \ref{descent}), establishing along the way their (connected components') semisimplicity (Corollary \ref{semisimple}). 

In \S 6, we use the work of Lafforgue and Drinfeld mentioned above to produce a collection $\{\rho'_{\lambda,s}\}_{\rho' \in \mc{S}}$ of $\lambda$-adic companions (which are necessarily integral) of these $\rho'_{\lambda',s}$. The semisimplicity of the monodromy groups is essential here too, in order to make use of \cite{drinfeld:pross}. From the collection of $\rho'_{\lambda,s}$ on $X_s$, we construct via tame specialization $G$-local systems $\rho'_{\lambda} \colon \pi_1^\top(X,x) \to G(\mc{O}_{\ol{K}_\lambda})$ on our original $X$ (over $\CC$). It is then shown that the $\#\mc{S}$ local systems $\rho'_\lambda$ constructed in this way are pairwise distinct and satisfy all of the defining properties of elements of $\mc{S}$. By counting, our original $\rho_{\lambda}$ must belong to this set, each member of which is integral at $\lambda$ (by virtue of arising from \'{e}tale local systems). We conclude that for all $\lambda \in \Sigma$, $\rho_{\lambda}$ can be conjugated into $G(\mc{O}_{\ol{K}_{\lambda}})$, and then we are done by the results of \S \ref{intconditionsection}. 

\section{Notation}
For a connected reductive group $G$, we let $G^{\text{der}}$ denote the derived group, $G^{\text{ad}}$ denote the adjoint group, $G^{\text{sc}}$ denote the simply-connected cover of $G^{\text{der}}$ (equivalently, of $G^\text{ad}$), and $Z_G$ denote the center. For a detailed treatment of how these constructions carry over to the case of the base scheme $\Spec(\ZZ)$, and indeed much more generally, we refer the reader to \cite{conrad:luminy} (or to the original constructions in \cite{sga3.3}), especially \cite[Theorem 3.3.4, Example 5.1.7, Theorem 5.3.1 and following]{conrad:luminy}.

If $X$ is a complex variety and $x \in X(\CC)$, we write $\pi_1^\top(X,x)$ to denote the fundamental group based at $x$ of the topological space $X(\CC)$ with its analytic topology. If $X$ is a scheme, and $x \colon \Spec(\Omega) \to X$ is a separably-closed base-point, we let $\pi_1^{\et}(X, x)$ denote the corresponding \'{e}tale fundamental group. If $j \colon X \hookrightarrow \ol{X}$ is a good compactification, so that $\ol{X}$ is regular and $D = \ol{X} \setminus X$ is a strict normal crossings divisor, we let $\pi_1^{\et, t}(X,x)$ be the quotient of the \'etale fundamental group corresponding to the fully faithful embedding of finite \'etale covers $Y \to X$ tamely ramified along $D$ (i.e.\ such that the valuation $v_{D_i}$ on $k(X)$ of any irreducible component $D_i$ of $D$ is tamely ramified in $k(Y)$) into the category of all finite \'etale covers of $X$. For a reference that this is a Galois category, we refer the reader to \cite[Theorem 2.4.2]{grothendieck:tame}. This group is independent of the compactification used, as seen from \cite[Theorem 1.2, Theorem 5.4]{kerz-schmidt:tameness}. 

If $K$ is a number field or a local field, we denote its ring of integers by $\mc{O}_K$. In the number field case, if $\Sigma$ is a set of finite places of $K$, we let $\mc{O}_{K, \Sigma}$ denote the localization away from $\Sigma$ of $\mc{O}_K$. We write $K_{\infty}$ for the product of the completions of $K$ at its infinite places, and we write $\mathbb{A}_K^{\infty}$ for the finite ad\`{e}les of $K$.

\section{Local integrality condition}\label{intconditionsection}
The next proposition gives the criterion we will apply in \S \ref{companionsection} to deduce the main theorem.
\begin{prop}\label{integralityprop}
Let $\Gamma$ be a finitely generated group, and let $G$ be a connected reductive group over $\ZZ$. Let $\rho \colon \Gamma \to G(\mc{O}_{K,\Sigma})$ be a homomorphism, where $\Sigma$ is a finite set of finite places of $K$. Assume that for each $\lambda \in \Sigma$, the representation $\rho_\lambda \colon \Gamma \to G(\ol{K}_\lambda)$ is $G(\ol{K}_\lambda)$-conjugate to a representation $\Gamma \to G(\mc{O}_{\ol{K}_\lambda})$. Then there exists a finite extension $L/K$ such that $\rho$ is $G(L)$-conjugate to a representation $\Gamma \to G(\mc{O}_L)$. 
\end{prop}
\begin{proof}
Each representation $\rho_\lambda \colon \Gamma \to G(\ol{K}_\lambda)$ is conjugate to a $G(\cO_{\ol{K}_\lambda})$-valued representation. That is, 
for each $\lambda \in \Sigma$, there is some $g'_\lambda \in G^\text{ad}(\ol{K}_\lambda)$ such that $g'_\lambda\rho {g_\lambda'}^{-1} \colon \Gamma \to G(\cO_{\ol{K}_\lambda})$ (also writing $g'_{\lambda}$ for any lift to $G(\ol{K}_{\lambda})$). 
The map $G^{\text{sc}}(\ol{K}_{\lambda}) \to G^{\text{ad}}(\ol{K}_{\lambda})$ is surjective, so we can find $g_\lambda \in G^\text{sc}(\ol{K}_\lambda)$ lifting $g'_\lambda$. 
It follows that the image of $g_\lambda$ in $G^\text{der}(\ol{K}_\lambda)$ conjugates $\rho_\lambda$ into a $G(\cO_{\ol{K}_\lambda})$-valued representation. We denote this representation $\rho_\lambda^{g_\lambda}$. 

Each $g_{\lambda}$ lies in $G^\text{sc}(L(\lambda))$ for some finite extension $L(\lambda)$ of $K_{\lambda}$, which we may assume to be Galois. In particular, $\rho_{\lambda}^{g_{\lambda}}$ has image contained in $G(L(\lambda)) \cap G(\mc{O}_{\ol{K}_{\lambda}})= G(\mc{O}_{L(\lambda)})$.
Since $\Sigma$ is finite, we can by class field theory find a Galois (in fact solvable) extension of number fields $L$ over $K$ such that for all $\lambda \in \Sigma$, and all places $\nu$ of $L$ above $\lambda$, $L_{\nu}$ is isomorphic (over $K_{\lambda}$) to $L(\lambda)$, and moreover $L$ has no real embeddings.

Consider the element $g_{\A_L^\infty} = (g_\nu) \in G^{\text{sc}}(\A_L^\infty)$ whose $\nu$-th component is, for $\nu$ above $\lambda \in \Sigma$, the image of $g_{\lambda}$ defined above under any $K_{\lambda}$-isomorphism $L(\lambda) \xrightarrow{\sim} L_{\nu}$,
and the identity otherwise. Then (the image in $G^\text{der}$ of) $g_{\A_L^\infty}$ conjugates the representation
    \[ \rho_{\A_L^\infty} \colon \Gamma \xrightarrow{\rho} G(\cO_{K,\Sigma}) \to G(K) \to G(L) \to G(\A_L^\infty) \]
into a homomorphism $\Gamma \to G(\prod \cO_{L_\nu})$. Now suppose that $\gamma_1, \ldots, \gamma_n \in G(\A_L^\infty)$ are the images under $\rho_{\A_L^\infty}$ of a set of generators of $\Gamma$. Let $\pi \colon G^\text{sc} \to G^\text{der}$ be the quotient map and consider the map 
    \[ \mathrm{conj} \colon G^\text{sc}(\A_L^\infty) \to \prod_{i =1}^n G(\A_L^\infty), \qquad g \mapsto (\pi(g)\gamma_1\pi(g)^{-1},\ldots, \pi(g)\gamma_n\pi(g)^{-1}). \]
Then $U = \mathrm{conj}^{-1}( G(\prod \cO_{L_\nu})^n)$ is open and non-empty since it contains $g_{\A_L^\infty}$. Since $G^{\text{sc}}$ is simply-connected, and $G^{\text{sc}}(L_{\infty})$ is non-compact, the strong approximation theorem \cite{kneser:strongapproximation} implies there is an element $g' \in G^\text{sc}(L) \cap U$. Then the element $\pi(g')$ is the desired element of $G(L)$.
\end{proof}

\section{Moduli of representations}\label{modulisection}
In this section $G$ is a split connected reductive group over a field $K$ of characteristic zero, with center $Z_G$ and adjoint group $G^\text{ad} = G/Z_G$. Let $\Gamma$ be a finitely-generated group with presentation
    \[ \Gamma = \langle r_1,\ldots, r_k \mid \{s_{\alpha}\}_{\alpha \in B} \rangle, \]
    where the set $B$ indexes the relations $s_{\alpha}$.
Let $b \colon G \to A$ be the maximal abelian quotient of $G$, and fix a homomorphism $\theta \colon \Gamma \to A(K)$. Let $\Rep_G(\Gamma, \theta)$ be the affine $K$-variety of representations of $\Gamma$ with abelianization equal to $\theta$, that is for a $K$-scheme $T$, $\Rep_G(\Gamma, \theta)(T)$ is the inverse image of $\Gamma \xrightarrow{\theta} A(K) \to A(T)$ under $\Hom(\Gamma, G(T)) \to \Hom(\Gamma, A(T))$. More explicitly, the map $\Rep_G(\Gamma, \theta)(T) \to G^k(T)$ given by $\rho \mapsto (\rho(r_1),\ldots, \rho(r_k))$ identifies $\Rep_G(\Gamma,\theta)$ with the closed subscheme of $G^k$ given by given by the conditions $s_{\alpha}(g_1,\ldots, g_k) = 1$ for $\alpha \in B$ and $b(g_i) = \theta(g_i)$ for $i = 1,\ldots, k$. (Note that since $G^k$ is Noetherian, finitely many of the relations $s_{\alpha}$ suffice to describe this subscheme.) There is an action of $G^\ad$ on $\Rep_G(\Gamma, \theta)$ by conjugating homomorphisms. Let $\Loc_G(\Gamma, \theta)$ be the resulting stack quotient,
    \[\Loc_G(\Gamma,\theta) = [\Rep_G(\Gamma,\theta)/G^\ad].\]
If $T$ is a $K$-scheme then $\Loc_G(\Gamma,\theta)(T)$ is the groupoid with objects $(\mc{E}, f)$ and isomorphisms $(\mc{E}',f') \xrightarrow{\sim} (\mc{E}, f)$ given respectively as commutative diagrams
\begin{center}
    \begin{tikzcd}
        \mathcal{E} \arrow[r,"f"] \arrow{d} & \Rep_G(\Gamma,\theta)  \\
        T & 
    \end{tikzcd}
    \qquad
    \begin{tikzcd}
        \mathcal{E}' \arrow[rr, bend left, "f'"] \arrow[r, dotted, "\simeq"] \arrow[rd] & \mathcal{E} \arrow[r,"f"] \arrow{d} & \Rep_G(\Gamma,\theta)  \\
         & T & 
    \end{tikzcd}
\end{center}
where $\mathcal{E} \to T$ is a (left) $G^\ad$-torsor, and $f\colon \mc{E} \to \Rep_G(\Gamma,\theta)$ is a $G^\ad$-equivariant morphism (likewise for $(\mc{E}', f')$). The isomorphism $\mc{E}' \to \mc{E}$ is required to be $G^\ad$-equivariant. 

\begin{defn}
If $T$ is a $K$-scheme, and $\rho$ a representation $\rho \colon \Gamma \to G(T)$ whose abelianization is $\theta$, we denote by $[\rho]$ in $\Loc_G(\Gamma, \theta)(T)$ the object $(G^\ad_T, f_\rho)$, where $f_\rho \colon G^\ad_T \to \Rep_G(\Gamma, \theta)$ is given by $f_\rho(g) = g\rho g^{-1}$. 
\end{defn}

\begin{lem}\label{local systems are conjugacy classes}
If $\Omega$ is an algebraically closed field (containing $K$), then $\Loc_G(\Gamma,\theta)(\Omega)$ can be identified with the groupoid of conjugacy classes of representations $\Gamma \to G(\Omega)$ with abelianization equal to $\theta$. 
\end{lem}
\begin{proof}
The groupoid $\mc{C}$ of conjugacy classes of morphisms $\Gamma \to G(\Omega)$ has objects being representations $\Gamma \to G(\Omega)$, whose abelianization is $\theta$. A morphism $\rho_1 \to \rho_2$ is given by inner automorphisms $\Ad(g) \colon G(\Omega) \to G(\Omega)$ for $g \in G^\ad(\Omega)$ such that $\rho_2 = \Ad(g) \circ \rho_1$, so there is a morphism only when $\rho_2 = g\rho_1 g^{-1}$ for some $g \in G^\ad(\Omega)$.

Given a representation $\rho \colon \Gamma \to G(\Omega)$ with abelianization $\theta$, we have the associated $G$-local system $ [\rho] = (G^\ad_\Omega, f_\rho)$ in $\Loc_G(\Gamma,\theta)(\Omega)$. The isomorphism $\rho \to g\rho g^{-1}$ for $g \in G^\ad(\Omega)$ gives a morphism of $G^\ad_\Omega$-torsors $R_{g^{-1}} \colon (G^\ad_\Omega, f_{\rho}) \to (G^\ad_\Omega, f_{g\rho g^{-1}})$ by $R_{g^{-1}}(h) = hg^{-1}$ on $\Omega$-points. This gives an isomorphism $[\rho] \to [g\rho g^{-1}]$ in $\Loc_G(\Gamma,\theta)$ since 
    \[ f_\rho(h) = h\rho h^{-1} = hg^{-1}g \rho g^{-1} gh^{-1} = f_{g \rho g^{-1}}(hg^{-1}). \]
Thus, $\rho \mapsto [\rho]$ is a functor $\mc{C} \to \Loc_G(\Gamma, \theta)(\Omega)$. 

To see that it is fully faithful, notice that both $\Hom(\rho_1, \rho_2)$ and $\Hom([\rho_1], [\rho_2])$ can naturally be identified the the set of $h \in G^\ad(\Omega)$ such that $\rho_2 = h\rho_1 h^{-1}$. 

We now show that the functor is fully faithful. By definition, an object of $\Loc_G(\Gamma, \theta)(\Omega)$ is a $G^\ad$-torsor $\mc{E} \to \Spec(\Omega)$ and a $G^\ad$-equivariant map $\varphi \colon \mc{E} \to \Rep_G(\Gamma,\theta)$. Since $\Omega$ is algebraically closed, there is a point $s \in \mc{E}(\Omega)$, which gives a $G^\ad$-equivariant isomorphism $\pi \colon G^\ad_\Omega \xrightarrow{\simeq} \mc{E}$ which is $g \mapsto gs$ on $\Omega$ points. Let $\rho = f(s)$. Since $\varphi(gs) = g\varphi(s)g^{-1}$, $\pi$ gives an isomorphism $[\rho] = (G^\ad, f_\rho) \xrightarrow{\simeq} (\mc{E}, \varphi)$. 
\end{proof}

We define a substack $\mr{IrrLoc}_G(\Gamma,\theta)$ of $\Loc_G(\Gamma,\theta)$ consisting of the objects of $\Loc_G(\Gamma,\theta)(S)$ whose base change along any geometric point $\Spec(\Omega) \to S$ yields via the identification of Lemma \ref{local systems are conjugacy classes} a conjugacy class of $G$-irreducible representations. Recall that a representation $\Gamma \to G(\Omega)$ is $G$-irreducible if the image is not contained in any proper parabolic subgroup. 

\begin{prop}
The substack $\mr{IrrLoc}_G(\Gamma,\theta) \subset \Loc_G(\Gamma,\theta)$ is open.
\end{prop}
\begin{proof}
For a $K$-scheme $S$ and $(\mc{E}, f) \in \Loc_G(\Gamma,\theta)(S)$ we say that $(\mc{E},f)$ is $G$-reducible at $s \in S$ if for some (equivalently any) a geometric point $\ol{s} \colon \Spec(\Omega) \to S$ lying over $s$, the conjugacy class of representations $\Gamma \to G(\Omega)$ corresponding to $(\mc{E}_{\ol{s}}, f_{\ol{s}})$ is $G$-reducible. We will show that the locus of $s \in S$ for which $(\mc{E},f)$ is $G$-reducible is closed. 

Fix a maximal torus and Borel subgroup $T \subset B \subset G$, with corresponding positive roots $\Phi^+$ and simple roots $\Delta$. For any finite set $S \subset \Delta$, let $P_S$ be the associated standard parabolic, and consider the bundle
    \[ \pi \colon Z = Z_S(\mc{E}, f) := G^\ad\backslash\left(\mc{E} \times (G/P_S) \right) \to S, \]
where $G^\ad$ acts diagonally (using that the left multiplication action of $G$ on $G/P_S$ factors through $G^\ad$). We get a $\Gamma$ action on $Z$ via the map $\mc{E} \to \Rep_G(\Gamma, \theta)$: if $x \in \mc{E}(U)$ maps via $f$ to $\rho_x \colon \Gamma \to G(U)$ in $\Rep_G(\Gamma,\theta)(U)$ then for $(x,y) \in (\mc{E} \times G/P_S)(U)$ define
    \[ \gamma\cdot(x, y) := (x, \rho_x(\gamma)y) \qquad \gamma \in \Gamma.  \]
This descends to an action on $Z(U)$ since if we change representatives to $(gx, gy)$ then 
    \[ \gamma\cdot(gx, gy) = (gx, \rho_{gx}(\gamma)gy) = (gx, g\rho_x(\gamma)g^{-1}gy) = (gx, g\rho_x(\gamma)y).\]


We claim that the locus $\pi(Z_S(\mc{E},f)^\Gamma) \subset S$ consists of all of the points $s \in S$ such that for any (equivalently every) geometric point $\ol{s} \colon \Spec(\Omega) \to S$ lying over $s$, the conjugacy class $\mr{im}(f_{\ol{s}} \colon \mc{E}_{\ol{s}}(\Omega) \to \Rep_G(\Gamma,\theta)(\Omega))$ contains a representation $\rho \colon \Gamma \to G(\Omega)$ whose image is contained in $P_S(\Omega)$. Indeed, given $s \in S$ and $\ol{s}$ a geometric point above $s$, we can write $\ol{s} = \pi([x,y])$ where $[x,y]$ is the class of $(x,y) \in (\mc{E}_{\ol{s}} \times G/P_S)(\Omega)$ in $Z_{\ol{s}}(\Omega)$. By multiplying $(x,y)$ by a suitable element of $G(\Omega)$, we can and do assume that $y = e$, the identity coset in $G/P_S(\Omega)$. Thus
\begin{align*}
\ol{s} \in Z_S(\mc{E},f)^\Gamma(\Omega) & \iff (x,e) = \gamma(x,e) = (x, \rho_x(\gamma)e) \text{ for all } \gamma \in \Gamma, \\
&\iff \rho_x(\gamma) \in \mr{Stab}_{G(\Omega)}(e) = P_S(\Omega) \text{ for all } \gamma \in \Gamma, \\
&\iff \mr{im}(f_{\ol{s}}(x) = \rho_x) \subset P_S(\Omega)
\end{align*}

It follows from this description that the locus of $s \in S$ where $(\mc{E}, f)$ is $G$-reducible is the set
    \[ \bigcup_{S \subsetneq \Delta} \pi(Z_S(\mc{E},f)^\Gamma) \]
For each $S$ the fixed point locus $Z_S(\mc{E},f)^\Gamma$ is closed, so $\pi(Z_S(\mc{E},f)^\Gamma)$ is closed (since $Z \to S$ is proper), and hence the finite union of such sets is closed. 
\end{proof}

Now suppose we are given $\gamma_1,\ldots, \gamma_N \in \Gamma$, and conjugacy classes $\mc{K}_i \subset G$ defined over $K$ for $i = 1,\ldots, N$. Note that these are locally closed subschemes of $G$. We consider the locally closed substack 
\begin{equation}\label{defn of moduli}
    \ul{M} = \ul{M}_G(\Gamma, \theta, \{(\gamma_i,\mc{K}_i)\}) \subset \mr{IrrLoc}_G(\Gamma, \theta),
\end{equation}
where an object $(\mc{E}, f) \in \mr{IrrLoc}_G(\Gamma,\theta)(T)$ is in $\ul{M}(T)$ if for every geometric point $\ol{t} \in T(\Omega)$, and any representation $\rho \colon \Gamma \to G^{\ad}(\Omega)$ with $[\rho] = (G^{\ad}_\Omega, f_\rho) \simeq (\mc{E}_{\ol{t}}, f_{\ol{t}})$, we have that $\rho(\gamma_i) \in \mc{K}_i(\Omega)$ for $i = 1,\ldots, N$. Note that this is independent of the choice of $\rho$, since any other $\rho'$ satisfying $(G^{\ad}_\Omega, f_{\rho'}) \simeq (\mc{E}_{\ol{t}}, f_{\ol{t}})$ will be $G^\ad(\Omega)$ conjugate to $\rho$ by lemma \ref{local systems are conjugacy classes}. 

To see that the stack $\ul{M}$ is locally closed in $\mr{IrrLoc}_G(\Gamma,\theta)$, argue as follows: Let $(\mc{E}, f)$ be an object of $\mr{IrrLoc}_G(\Gamma,\theta)(T)$, and let $T_0 \subset T$ be set of points $t \in T_0$ such that for any geometric point $\ol{t} \in T(\Omega)$ over $t$, $(\mc{E}_{\ol{t}}, f_{\ol{t}}) \in \ul{M}(\Omega)$. Let $\mc{E}_0 = \mc{E} \times_T T_0$. This is $G^\ad$ stable, so it suffices to show that $\mc{E}_0$ is locally closed. Let $\rho \colon \Gamma \to G(\mc{E})$ be the representation associated to $f$. Each $\rho(\gamma_i)$ is a map $\mc{E} \to G$ and
    \[ \mc{E}_0 = \bigcap_{i =1}^N \rho(\gamma_i)^{-1}(\mc{K}_i). \]
Since each $\mc{K}_i$ is locally closed in $G$, the intersection is locally closed. 
In summary, we have the following proposition.

\begin{prop}\label{moduli stack is algebraic}
Let $\Gamma$ be a finitely-generated group, $K$ a field, $G$ a split connected reductive group over $K$ with maximal abelian quotient $A$. Fix
\begin{itemize}
    \item $\theta \colon \Gamma \to A(K)$ a representation; \item $\gamma_1, \ldots, \gamma_N \in \Gamma$; \item conjugacy classes $\mc{K}_i \subset G$ defined over $K$ (for $i = 1,\ldots, N$).
\end{itemize} 

\begin{rmk}
The case of interest for this paper is when $K$ is a number field. 
\end{rmk}
Then the stack $\ul{M}$ of conjugacy classes of $G$-irreducible representations $\rho \colon \Gamma \to G$ with abelianization $\theta$, such that $\rho(\gamma_i) \in \mc{K}_i$ for all $i=1, \ldots, N$ is an algebraic stack of finite type over $K$. 
\end{prop}

Now suppose $\rho_0 \colon \Gamma \to G(K)$ is a $G$-irreducible representation such that $\rho_0 (\gamma_i) \in \mc{K}_i(K)$. The $G$-local system $[\rho_0] = (G^\ad, f_{\rho_0})$ associated to $\rho_0$ is then an object of $\ul{M}(K)$ (for $\theta = \rho_0^\text{ab}$).

\begin{prop}\label{tangent space group cohomology}
Let $\fg^\mr{der}$ be the Lie algebra of $G^\mr{der}$. Then the Zariski tangent space $T_{[\rho_0]}\ul{M}$ is the kernel of 
    \[ H^1(\Gamma, \fg^\mr{der}(K)) \xrightarrow{\mr{res}} \bigoplus_{i=1}^n H^1(\gamma_i^\ZZ, \fg^\mr{der}(K)). \]
\end{prop}
\begin{proof}
For any point $(\mc{E},f) \in \ul{M}(K[\varepsilon])$, let $(\mc{E}_0, f_0) = i^\ast(\mc{E},f)$ with $i$ the closed immersion $\Spec(K) \to \Spec(K[\varepsilon])$. The tangent space $T_{[\rho_0]\ul{M}}$ consists of the points $v = \left(\mc{E}, f\right)$ in $\ul{M}(K[\varepsilon])$ such that $(\mc{E}_0, f_0) \xrightarrow{\simeq} (G^\ad, f_{\rho_0})$ over $K$ (the isomorphism not being part of the data).
Fix one such isomorphism $\varphi_0$. By the formal criterion of smoothness, the section $\Spec(K) \xrightarrow{} G^\ad \xrightarrow{\varphi_0^{-1}} \mc{E}_0$ extends to give a section of $\mc{E} \to \Spec(K[\varepsilon])$, hence $\mc{E}$ is trivial. By choosing a trivialization $\varphi \colon \mc{E} \xrightarrow{\sim} G^\ad_{\Spec(K[\varepsilon])}$ (extending $\varphi_0$ on the closed fiber) we get a representation $\rho = f\circ \varphi^{-1}(1) \colon \Gamma \to G(K[\varepsilon])$ whose composition with $G(K[\varepsilon]) \to G(K)$ is $\rho_0$, and which also maps $\gamma_i$ to an element of $\mc{K}_i(K[\varepsilon])$.

The map $\alpha(\gamma) = \rho(\gamma)\rho_0(\gamma)^{-1}$ gives an element $\alpha \in Z^1(\Gamma, \fg^\mr{der}(K))$ (the ``fixed determinant" condition ensures that $\alpha$ is valued in $\fg^{\mr{der}}(K) \subset \fg(K)$). The manipulation  
\begin{align*}
    \rho(\gamma\beta)\rho_0(\gamma\beta)^{-1} &= \rho(\gamma)\rho(\beta)\rho_0(\beta)^{-1}\rho_0(\gamma)^{-1}, \\
        &= \rho(\gamma)\rho_0(\gamma)^{-1} \cdot \mathrm{Ad}( \rho_0)(\gamma)\left(\rho(\beta)\rho_0(\beta)^{-1}\right).
\end{align*} 
verifies the cocycle condition, and it is similarly easy to see that changing the trivialization $\varphi$ modifies $\alpha$ by a coboundary.

Therefore there is a well defined map $T_{[\rho_0]}\ul{M} \to H^1(\Gamma, \fg^\mr{der}(K))$. It is injective because one can reconstruct the point $v$ from the class $[\alpha]$ by taking $\rho = \alpha \rho_0$, and then $v = [\rho]$. On the other hand, given a cocycle $\alpha$ the point $[\alpha \rho_0]$ is in $T_{[\rho_0]}\ul{M}$ if and only if for each $i = 1,\ldots, n$ we have $\alpha\rho_0(\gamma_i) \in \mc{K}_i(K[\varepsilon])$. The restriction $\mr{res}([\alpha])$ is trivial if and only if for each $i$ there is some $\xi_i \in \fg^\mr{der}(K)$ such that $\alpha(\gamma_i) = \xi_i\rho_0(\gamma_i)\xi_i^{-1}\rho_0^{-1}(\gamma_i)$, hence
    \[ \alpha\rho_0(\gamma_i) = \xi_i\rho_0(\gamma)\xi_i^{-1}, \]
and thus $\alpha\rho_0(\gamma_i)$ belongs to $\mc{K}_i(K[\varepsilon])$. We conclude that the image of $T_{[\rho_0]}\ul{M}$ in $H^1(\Gamma, \fg^{\mr{der}}(K))$ is the kernel of $\mr{res} \colon H^1(\Gamma, \fg^\mr{der}(K)) \to \oplus_{i=1}^n H^1(\gamma_i^\ZZ, \fg^\mr{der}(K))$. 
\end{proof}

\subsection{Back to geometry}\label{back to geometry}
We now specialize to the case of interest. Let $X$ be a smooth connected quasi-projective complex variety, fix a base-point $x \in X(\CC)$, and let $j \colon X \hookrightarrow \ol{X}$ be a good compactification with $D = \ol{X}\setminus X$ the strict normal crossings divisor written as the union $D = \cup_{i=1}^N D_i$ of its irreducible components $D_i$. Let $D^\text{sing}$ be the singular locus of $D$ and $U = \ol{X} \setminus D^\text{sing}$ so there is a factorization $j \colon X \xhookrightarrow{a} U \xhookrightarrow{b} \ol{X}$. For $i = 1,\ldots, N$, fix $y_i \in D_i \cap U$, $\Delta_i \subset U$ a small open ball around $y_i$, and $x_i \in \Delta_i^\times= \Delta_i \setminus (D_i \cap \Delta_i)$. Choose $T_i \in \pi_1^\top(X,x)$ that generates the image of $\ZZ \cong \pi_1^\top(\Delta_i^\times,x_i) \to \pi_1^\top(X,x)$ (with the map depending on a fixed choice of path from $x$ ending in $\Delta_i^\times$. 

Let $K$ be a characteristic zero field. Suppose $\mc{K}_i$ is a conjugacy class defined over $K$ for each $i = 1, \ldots, N$, 
fix a character $\theta \colon \pi_1^{\mr{top}}(X, x) \to A(K)$, and define 
    \[ \ul{M} := \ul{M}_G(\pi_1^{\top}(X,x),\theta, \{(T_i, \mc{K}_i)\}) \subset \mr{Loc}_G(\pi_1^\top(X,x),\theta), \]
the moduli stack of $G$-irreducible $G$-local systems on $X$ with abelianization $\theta$ and monodromy $\mc{K}_i$ around $D_i$, as constructed above. Let $\rho_0 \colon \pi_1^\top(X,x) \to G(K)$ be a homomorphism that is $G$-irreducible, has abelianization $\theta$, and has monodromy $\rho_0(T_i) \in \mc{K}_i(K)$. Then the associated $G$-local system $[\rho_0]$ gives an element of $\ul{M}(K)$. Let $\ul{\fg}^\mr{der}$ be the locally constant sheaf on $X$ corresponding to the representation 
\[ \pi_1^\top(X,x) \xrightarrow{\rho_0} G(K) \xrightarrow{\Ad} \GL(\fg^\mr{der}(K)). \]

\begin{prop}\label{tangent space}
The tangent space of $\ul{M}$ at the point $[\rho_0]$ is the finite-dimensional $K$-vector space $H^1(U, a_\ast \ul{\fg}^\mr{der})$. 
\end{prop}
\begin{proof}
Giving the Lie algebra $\fg^{\mr{der}}(K)$ the structure of a $\pi_1(X,x)$-module vial $\Ad \circ \rho_0$, there is a commutative diagram
\begin{center}
    \begin{tikzcd}
        0 \arrow[r] & T_{[\rho_0]}\ul{M} \arrow[r] & H^1(\pi_1^\top(X), \fg^{\mr{der}}(K)) \arrow[r, "\mr{res}"] \arrow[d, "\simeq"] & \oplus_{i=1}^N H^1(\pi_1(\Delta_i^\times), \fg^{\mr{der}}(K)) \arrow[d, "\simeq"] \\
         & & H^1(X, \ul{\fg}^\mr{der}) \arrow[r, "\mr{res}"] & \oplus_{i=1}^N H^1(\Delta_i^\times, \ul{\fg}^\mr{der}|_{\Delta_i^\times}) 
    \end{tikzcd}
\end{center}
and Proposition \ref{tangent space group cohomology} shows that the top row is exact. The result follows from the identification of the kernel of the map on the bottom row with $H^1(U, a_\ast \ul{\fg}^\mr{der})$ (which can be seen using the Leray spectral sequence for $a$). 
\end{proof}

\begin{rmk}\label{intermediate extension}
As was noted in \cite[Remark 2.4]{esnault-groechenig:rigid}, we have $H^1(U, a_\ast \ul{\fg}^{\mr{der}}) = H^1(\ol{X}, j_{!\ast} \ul{\fg}^{\mr{der}})$, where $j_{!\ast} \ul{\fg}^{\mr{der}}$ is the intermediate extension.
\end{rmk}

\section{Arithmetic descents}\label{descentssection}
\subsection{Reduction to the adjoint case}
We begin the proof of Theorem \ref{mainthm} by reducing to the case in which $G$ is a simple (split, as we always assume) adjoint group:
\begin{lem}\label{adreduction}
It suffices to prove Theorem \ref{mainthm} in the case where $G$ is a simple adjoint group.
\end{lem}
\begin{proof}
Consider $\rho$ as in Theorem \ref{mainthm}. The projectivization $\mathbb{P}(\rho) \colon \pi_1^{\mr{top}}(X, x) \xrightarrow{\rho} G(\CC) \to G^{\mr{ad}}(\CC)$ is $G^{\mr{ad}}$-irreducible, $G^{\mr{ad}}$-cohomologically rigid, and has quasi-unipotent local monodromy (as is evident from the definitions). Note that $\mathbb{P}(\rho)$ having these properties is equivalent to its projections to each simple factor of $G^{\mr{ad}}$ having these properties. Thus, under the assumption of the Lemma, we may assume that after replacing $\rho$ by a $G(\CC)$-conjugate, $\mathbb{P}(\rho)$ factors through $G^{\mr{ad}}(\mc{O}_K)$ for some number field $K$. We claim that $\rho$ itself factors through $G(\mc{O}_L)$ for some finite extension $L/K$. Indeed, $\rho$ of course factors through the preimage of $\mr{im}(\mathbb{P}(\rho))$ in $G(\CC)$. By our assumption that $\rho$ has finite image in $A(\CC)$, $\im(\rho)$ is further contained in $G^{\mr{der}}(\CC) \cdot Z_G(\ol{\ZZ})$, since any point of finite order in $Z_G(\CC)$ is contained in $Z_G(\ol{\ZZ})$.
Thus it suffices from the long-exact sequence in cohomology to check that the image of each of the generators of the finitely-generated group $\pi_1^{\mr{top}}(X, x)$ in $G^{\mr{ad}}(\mc{O}_K)$ is itself in the image of $G^{\mr{sc}}(\mc{O}_L)\to G^{\mr{ad}}(\mc{O}_L)$ for some finite extension $L/K$. 
Since $\ker(G^{\mr{sc}} \to G^{\mr{ad}})$ is a product (over the simple factors) of group schemes $\mu_r$, it suffices to find $\mc{O}_L$ trivializing any given class in $H^1_{\mr{fppf}}(\Spec(\mc{O}_K), \mu_r)$. The (fppf) Kummer sequence yields a short-exact sequence
\[
\mc{O}_K^{\times}/(\mc{O}_K^{\times})^r \to H^1_{\mr{fppf}}(\Spec(\mc{O}_K), \mu_r) \to \mr{Pic}(\mc{O}_K)[r].
\]
The desired $L$ exists because every element of the class group of $\mc{O}_K$ becomes trivial after passage to some finite extension $\mc{O}_L$, and we can adjoin an $r^{th}$ root to trivialize an element of $\mc{O}_K^\times/(\mc{O}_K^\times)^r$. Finally, semisimplicity of the connected monodromy of $\mathbb{P}(\rho)$ clearly implies that of $\rho$ itself, as above using that $\rho$ has finite image in $A(\CC)$.
\end{proof}
In this section and \S \ref{companionsection}, we will therefore prove:
\begin{thm}\label{mainthm adjoint case}
Theorem \ref{mainthm} is true if $G$ is simple and adjoint. 
\end{thm}

\subsection{Algebraicity of the local system}


Let $X$ be a smooth quasi-projective variety over $\CC$ and $(\ol{X}, U, D, \ldots)$ a compactification as in the first paragraph of \ref{back to geometry}. Let $G$ be a split simple adjoint group over $\ZZ$, and let $\rho_0 \colon \pi_1^\top(X,x) \to G(\CC)$ be a representation defining a $G$-local system on $X$ such that 
\begin{itemize}
    \item $\rho_0$ is $G$-cohomologically rigid, i.e.\ $H^1(U, a_\ast \ul{\fg}^\mr{der}) = 0$. 
    \item $\rho_0$ is $G$-irreducible. 
    \item The local monodromies (i.e.\ $\rho_0(T_1), \ldots, \rho_0(T_N)$ with $T_i$ a small enough loop around $D_i$ as in \ref{back to geometry}) are quasi-unipotent.
\end{itemize}

Richardson showed (\cite[Theorem 3.6, Theorem 4.1]{richardson:conjugacyntuple}, valid over fields of characteristic zero) that $G$-irreducibility of a closed subgroup $H \subset G$ is equivalent to $H$ being reductive with centralizer $C_G(H)$ that is finite modulo the center $Z_G$ of $G$ (in our case, simply finite). In particular, $\ol{\im(\rho_0)}$ enjoys these two properties.



Our strategy, following \cite{esnault-groechenig:rigid}, to prove Theorem \ref{mainthm adjoint case} is to define a finite set $\mc{S}$ of $G$-local systems with $[\rho_0] \in \mc{S}$, and to prove that all elements of $\mc{S}$ are integral. For this, choose an integer $h \geq 1$ such that $\rho_0(T_i)$ is quasi-unipotent of index\footnote{The index of a quasi-unipotent element $g\in G(\CC)$ is the order of $g_s$ is the Jordan decomposition $g = g_sg_u$.} dividing $h$ for all $i = 1,\ldots, N$. Let $\mc{S} = \mc{S}(G,h)$ be the set of conjugacy classes of representations $\rho \colon \pi_1^\top(X,x) \to G(\CC)$ that are $G$-irreducible, $G$-cohomologically rigid, and such that each $\rho(T_i)$ is quasi-unipotent of index dividing $h$. 

\begin{lem}
The set $\mc{S}$ is finite, and there exist a number field $K \subset \CC$ and a finite set $\Sigma$ of finite places of $K$ such that each element of $\mc{S}$ is represented by some representation $\rho \colon \pi_1^\top(X,x) \to G(\mc{O}_{K, \Sigma})$.
\end{lem}
\begin{proof}
There is a finite collection $\mc{K}_1,\ldots, \mc{K}_M \subset G(\CC)$ of conjugacy classes of quasi-unipotent elements of index dividing $h$. Indeed, one can choose a finite number of semisimple elements 
that represent any semisimple conjugacy class of elements of order dividing $h$ 
(the $h$-torsion in $T(\CC)$, for $T$ a maximal torus of $G$ will do). The conjugacy classes with a fixed semisimple part $s$ correspond to unipotent conjugacy classes in $C_G(s)$. As $C_G(s)^0$ is connected reductive (since $s$ is semisimple), it has finitely many unipotent conjugacy classes, and \textit{a fortiori} the same then holds for $C_G(s)$ (all unipotent elements of $C_G(s)$ belong to $C_G(s)^0$).

The $\mc{K}_i$ are the complex points of a locally closed algebraic subvariety of $G_K$ for some common number field $K$. By abuse of notation, we write $\mc{K}_i \subset G_K$ for this variety. 
For each $N$-tuple $J = (j_1,\ldots, j_N)$ with $1 \leq j_1, \ldots, j_N \leq M$,
let $\ul{M}_J = \ul{M}_G(\pi_1^\top(X,x), \{(T_i, \mc{K}_{j_i})\})$ be the moduli space defined in (\ref{defn of moduli}) of \S \ref{modulisection}, that is the collection of $G$-irreducible $G$-local systems on $X$ with monodromy around $D_i$ in $\mc{K}_{j_i}$. Note that since $G$ is adjoint, the condition fixing the projection $G \to A=\{1\}$ is vacuous. Let $\ul{N} = \cup_J \ul{M}_J$ as $J$ runs over the finite set of $N$-tuples as above. By Proposition \ref{tangent space} the set $\mc{S}(G,h)$ is contained in the isolated points of the stack $\ul{N}$, and since $\ul{N}$ is finite type by Proposition \ref{moduli stack is algebraic}, there are finitely many isolated points. 

More precisely, suppose $\rho \colon \pi_1^\top(X,x) \to G(\CC)$ has $\rho(T_i) \in \mc{K}_{j_i}$ for $i = 1,\ldots, N$. If $[\rho] \in \mc{S}(G,h)$ we claim that the image of the map $f = f_\rho \colon G_{\CC} \to \Rep_{G}(\pi_1^\top(X,x), \{T_i, \mc{K}_{j_i} \})$, given by $g \mapsto g\rho g^{-1}$, is a connected component. If not, $\mr{im}(f)$ is contained in some closed connected subspace $Z \supsetneq \mr{im}(f)$. Let $v \in T_{f(1)} Z \setminus T_{f(1)}\mr{im}(f_\rho)$. Then $v$ corresponds to a cocycle $\alpha_v \in Z^1(\pi_1^\top(X,x), \fg^{\mr{der}}(\CC))$, with action $\mr{Ad}\circ \rho$ on $\fg^{\mr{der}}(\CC)$. The cocycles from $T_{f(1)}\mr{im}(f)$ are precisely the coboundaries, so our choice of $v$, along with Proposition \ref{tangent space} shows that the cohomology class $[\alpha_v]$ is a non zero element of $H^1(U, a_\ast \ul{\fg^{\mr{der}}(\CC)})$. This contradicts the fact that $\rho$ is $G$-cohomologically rigid. Thus $\mr{im}(\rho)$ is a connected component. As $\Rep_{G}(\pi_1^\top(X,x), \{T_i, \mc{K}_{j_i} \})_{\CC}$ is finite type over $\CC$, there are finitely many connected components. We conclude that $\mc{S}(G,h)$ is finite as it injects into the set of connected components of finitely many representation varieties of the above type. 


The group $\mr{Aut}(\CC/\ol{\QQ})$ acts on the isomorphism classes of objects of $\ul{N}(\CC)$, and permutes the finite set $\mc{S}(G,h)$. We claim that if $\rho \colon \pi_1^\top(X,x) \to G(\CC)$ represents a $G$-local system $[\rho] \in \mc{S}(G,h)(\CC)$, then $\rho$ is defined over $\ol{\QQ}$ (and hence over a number field $K$ since $\pi_1^{\top}(X, x)$ is finitely-generated). This can be seen using the pseudo-character $\mr{tr}\, \rho = (\Theta_n)_{n \geq 1}$ of $\rho$ (for background on pseudo-characters, including the notation we use, see \cite[\S 4]{bhkt:fnfieldpotaut}; the notion is due to V. Lafforgue in \cite{vlafforgue}). Here each $\Theta_n$ is the map $\Theta_n \colon \ZZ[G^n]^G \to \mr{Map}(\pi_1^\top(X,x)^n, \CC)$ given by $\Theta_n(f)(\gamma_1,\ldots, \gamma_n) = f(\rho(\gamma_1),\ldots, \rho(\gamma_n))$. The pseudo-character of ${}^\sigma\rho$ is then $({}^\sigma\Theta_n(f))(\ul{\gamma}) = \sigma(\Theta_n(f)(\ul{\gamma}))$ for $\ul{\gamma} = (\gamma_1,\ldots, \gamma_n) \in \pi_1^\top(X,x)^n$. If $\rho$ cannot be defined over $\ol{\QQ}$, then by \cite[Proposition 11.7]{vlafforgue} (see also \cite[Theorem 4.5]{bhkt:fnfieldpotaut}) there is some $f \in \ZZ[G^n]^G$ and some $\ul{\gamma} \in \pi_1^\top(X,x)^n$ such that $\Theta_n(f)(\ul{\gamma})$ is transcendental over $\ol{\QQ}$. It follows that there is an infinite set $H \subset \mr{Aut}(\CC/\ol{\QQ})$ such that $({}^\sigma\Theta_n(f))(\ul{\gamma}) \neq ({}^\tau\Theta_n(f))(\ul{\gamma})$ for any $\sigma \neq \tau \in H$, and hence by \textit{loc. cit.}, ${}^\sigma \rho$ is not conjugate to ${}^\tau\rho$. Thus the set of conjugacy classes of $\{{}^\sigma\rho\}_{\sigma \in H}$ is infinite and contained in $\mc{S}(G,h)$ which contradicts finiteness of $\mc{S}(G,h)$. 

Choose representatives $\rho_i, i = 1,\ldots, \mathfrak{N}$ of the conjugacy classes in $\mc{S}(G,h)$. By the above discussion each $\rho_i$ is defined over a number field $K$. Enlarging $K$, we can assume each $\rho_i$ is defined over the same $K$. Since $\pi_1^\top(X,x)$ is finitely generated, each $\rho_i$ is defined over $\mc{O}_{K,\Sigma}$ for $\Sigma$ a finite set of finite places. Enlarging $\Sigma$, we may assume that $\rho_1, \ldots, \rho_{\mathfrak{N}}$ are all defined over $\mc{O}_{K,\Sigma}$.
\end{proof}

By the previous lemma we choose once and for all representations
    \[ \rho_i \colon \pi_1^{\mathrm{top}}(X,x) \to G(\cO_{K,\Sigma}), \qquad \qquad i = 1, \ldots, \mathfrak{N}, \]
which form a complete set of representatives for the conjugacy classes in $\mc{S}(G,h)$. 


Fix $\lambda \in \Spec(\cO_{K,\Sigma})$ a finite place. Consider the representation
    \[ \rho_{i,\lambda}^{\circ,\mathrm{top}} \colon \pi_1^{\mathrm{top}}(X,x) \to G(\cO_{K,\Sigma}) \to G(\cO_{K_\lambda}). \]
Composing with $G(\cO_{K_\lambda}) \to G(K_\lambda)$ we get a representation $\rho_{i,\lambda}^{\mathrm{top}}$ (and similarly we get representations $\rho_{i,\lambda'}^\top$ for $\lambda' \in \Sigma$). Since $G(\cO_{K_\lambda})$ is a profinite group, the representations $\rho_{i,\lambda}^{\circ,\mathrm{top}}, \rho_{i,\lambda}^\mathrm{top}$ of $\pi_1^{\mathrm{top}}(X,x)$ factor through $\pi_1^\et(X,x)$. We denote the factored representations by $\rho_{i,\lambda}^\circ$ and $\rho_{i,\lambda}$. 

The next step is to spread $X$ out over a scheme $S$ of finite type over $\Spec(\ZZ)$ and then to specialize to characteristic $p$ as in \cite{esnault-groechenig:rigid}. More precisely, there exists $S$ a connected regular scheme of finite type over $\Spec(\ZZ)$ with a complex generic point $\eta \colon \Spec(\CC) \to S$ such that $(j \colon X \hookrightarrow \ol{X}, x, D, D_i)$ is the fiber over $\eta$ of $(j_S \colon X_S \hookrightarrow \ol{X}_S, x_S, D_S, D_{i,S})$, where $\ol{X}_S$ is smooth and projective over $S$, $D_S = \cup_{i} D_{i,S} \subset \ol{X}_S$ is a relative strict normal crossings divisor whose complement is $X_S$, and $x_S \in X_S(S)$. If $s \in S(k)$ let $(j_{{s}} \colon X_{{s}} \to \ol{X}_{{s}}, x_{{s}}, D_{{s}}, D_{i,{s}})$ be the data over $k$ obtained as the base change of the corresponding data over $S$ via $s$. 

\begin{rmk}\label{Spreading out the curve}
Strictly speaking, for the proof of Proposition \ref{companion properties} we will need a little more: Fix a curve $\ol{C} \subset \ol{X}$ that is smooth and projective, does not meet $D^\mr{sing}$ and meets $D\setminus D^{\mr{sing}}$ transversely (for the existence of $\ol{C}$, see \cite[Theorem 6.3]{jouanolou:bertini}). We take $x \in \ol{C}$, and choose points $y_i \in D_i \cap \ol{C}$. We may (and do) choose $S$ such that $(j \colon X \hookrightarrow \ol{X}, x, D, D_i, \ol{C}, y_i)$ spreads out over $S$ to $(j_S \colon X_S \hookrightarrow \ol{X}_S, x_S, D_S, D_{i,S}, \ol{C}_S, y_{i,S})$ satisfying the above constraints on $j_S, \ol{X}_S, D_S, D_{i,S}$ and such that $\ol{C}_S$ is smooth and projective over $S$ and $C_S$ intersects $D_S$ transversely. Write $\ol{C}_s$ and $y_{i,s}$ for the fiber of $\ol{C}_S$ and $y_{i,S}$ over $s$. The reader may safely ignore the curve $\ol{C}$ until the proof of Proposition \ref{companion properties}. 
\end{rmk}

We now choose a finite field $k$ and a closed point $s \in S(k)$ such that $p = \mr{char}\, k$ is coprime to
\begin{itemize}
    \item the prime $\ell$ under $\lambda$;
    \item all places under $\Sigma$; 
    \item for each $i$, the cardinality of the image of the reduction 
    \[    
    \bar{\rho}^\circ_{i, \lambda} \colon \pi_1^\et(X,x) \to G(k(\lambda))
    \]
    of $\rho^\circ_{i, \lambda}$ to the residue field $k(\lambda)$ of $\cO_{K_{\lambda}}$; and
    \item the index of quasi-unipotence of the $\rho_i$. 
\end{itemize}

Given these choices, the $p$-part of $\pi_1^\et(X,x)$ is in the kernel of $\rho_{i,\lambda}^\circ$, hence $\rho_{i,\lambda}^\circ$ factors through the prime-to-$p$ fundamental group $\pi_1^\et(X,x) \to \pi_1^{\et,p'}(X,x)$. If $\ol{s}$ is a geometric point over $s$, there is a discrete valuation ring $R$ and a map $\Spec(R) \to S$ mapping the generic point of $\Spec(R)$ to $\eta$ and the special point to $\ol{s}$. 
By pulling back $X_S$ to $\Spec(R)$, there is a prime-to-$p$ specialization map\footnote{There is also a surjective tame specialization map $\pi_1^\et(X,x) = \pi_1^{\et,t}(X,x) \to \pi_1^{\et, t}(X_{\ol{s}},x_{\ol{s}})$ (see \cite[Corollary A.12]{lieblich-olsson-pi1}) which we will use later (particularly in the proof of Proposition \ref{companion properties}).} on prime-to-$p$ fundamental groups giving an isomorphism $\Sp \colon \pi_1^{\et,p'}(X,x) \xrightarrow{\sim} \pi_1^{\et,p'}(X_{\ol{s}}, x_{\ol{s}})$ (see \cite[Corollary A.12]{lieblich-olsson-pi1}, or originally \cite{sga1}). Thus we get representations
    \[ \rho_{i,\lambda,\ol{s}}^\circ \colon \pi_1^\et(X_{\ol{s}}, x_{\ol{s}}) \twoheadrightarrow \pi_1^{\et,p'}(X_{\ol{s}}, x_{\ol{s}}) \xrightarrow{\mr{Sp}^{-1}} \pi_1^{\et,p'}(X,x) \xrightarrow{\rho_{i,\lambda}^{\circ}} G(\cO_{K_\lambda}) \]
and again denote by
    \[ \rho_{i,\lambda, \ol{s}} \colon \pi_1^{\et,p'}(X_{\ol{s}}, x_{\ol{s}}) \to G(\cO_{K_\lambda}) \to G(K_{\lambda}) \]
the composites into $G(K_{\lambda})$.

\begin{rmk}\label{quasi-unipotent monodromy}
We say that $\rho_{i, \lambda, \ol{s}}$ has quasi-unipotent local monodromy with index dividing $h$ along $D_j$ ($j=1, \ldots, N$) if it maps any pro-generator of the tame inertia group at $D_j$ (i.e., the absolute Galois group of the fraction field of the strict henselization of the DVR $\mc{O}_{\ol{X}_{\bar{s}}, D_{j, \bar{s}}}$) to a quasi-unipotent element whose $h^{th}$ power is unipotent. This property is inherited from the corresponding property of $\rho_i$. Indeed, the monodromy of $\rho_i$ at $D_j$ is the same as the monodromy of $\rho_i|_{C}$ at $y_j$ for our chosen curve $C$ above. Similarly, the the monodromy of $\rho_{i,\lambda, \ol{s}}$ at $D_{j,\ol{s}}$ is the same as the monodromy of it's restriction $\rho_{i,\lambda,\ol{s}}|_{C_{\ol{s}}}$ to $C_{\ol{s}}$ at $y_{i,\ol{s}}$ by Abhyankar's lemma \cite[XIII Proposition 5.2]{sga1} (noting that these local systems are tamely ramified). Thus we reduce to showing that the monodromy of $\rho_{i,\lambda,\ol{s}}|_{C_{\ol{s}}}$ at $y_{i,\ol{s}}$ is quasi-unipotent of index dividing $h$, which follows from \cite[Lemma 3.2]{esnault-groechenig:rigid} (or \cite[XIV 1.1.10]{sga7.2}).
\end{rmk}

\subsection{Arithmetic descent}
We would like to say that the $G$-local system $[\rho_{i,\lambda,\ol{s}}]$ on $X_{\ol{s}}$ is the pullback of a local system on $X_{s}$, or equivalently that there exists a representation $\rho_{i,\lambda, s}$ making the diagram commute
\begin{center}
    \begin{tikzcd}
        \pi_1^\et(X_{\ol{s}}, x_{\ol{s}}) \arrow[rd, "\rho_{i,\lambda, \ol{s}}"] \arrow[d, hook] & \\
        \pi_1^\et(X_s, x_{\ol{s}}) \arrow[r, "\rho_{i,\lambda, s}"] & G(K_\lambda) 
    \end{tikzcd}
\end{center}
The next proposition (a variant on \cite[Proposition 3.1]{esnault-groechenig:rigid}, \cite[Theorem 4]{simpson:higgs}) says that this indeed can be arranged, after possibly replacing $k$ by a finite extension. 
\begin{prop}\label{descent}
After replacing $k$ by a finite extension there is a representation 
    \[ \rho_{i,\lambda, s} \colon \pi_1^{\et, p'}(X_s, x_{\ol{s}}) \to G(\ol{\QQ}_\lambda) \]
whose restriction to $\pi_1^\et(X_{\ol{s}}, x_{\ol{s}})$ is $\rho_{i,\lambda,\ol{s}}$ and whose image has the same Zariski-closure as $\rho_{i, \lambda, \bar{s}}$.
\end{prop}
\begin{proof}

The point $x_s \in X_s$ is rational, so it splits the ``prime-to-$p$ homotopy sequence"
    \[ 1 \to \pi_1^{\et, p'}(X_{\ol{s}}, x_{\ol{s}}) \to \pi_1^{\et,p'}(X_s,x_s) \to \pi_1^{\et}(s, \ol{s}) \to 1, \]
which is by definition the pushout of the homotopy exact sequence for $\pi_1^{\et}(X_s, x_s)$ along $\pi_1^{\et}(X_{\bar{s}}, x_{\bar{s}}) \to \pi_1^{\et, p'}(X_{\bar{s}}, x_{\bar{s}})$.
For $\sigma \in \pi_1^\et(s,\ol{s})$ we also denote by $\sigma$ its image in $\pi_1^{\et, p'}(X_s, x_s)$ under the splitting from $x_s$. Let
    \[ \rho_{i,\lambda}^\sigma \colon \pi_1^\et(X,x) \to G(K_\lambda), \qquad \gamma \mapsto \rho_{i,\lambda,\ol{s}}(\sigma\cdot \mr{Sp}(\gamma)\cdot \sigma^{-1}), \]
where $\mr{Sp} \colon \pi_1^\et(X,x) \to \pi_1^{\et, p'}(X_{\ol{s}}, x_{\ol{s}})$ is the specialization map. We claim that the $G$-local system $[\rho_{i,\lambda}^{\sigma,\top}]$ represented by $\rho_{i,\lambda}^{\sigma,\top}$, the pullback of $\rho^\sigma_{i, \lambda}$ to $\pi_1^{\top}(X, x)$, has quasi-unipotent local monodromies of index dividing $h$ (i.e. $[\rho_{i,\lambda}^{\sigma, \top}] \in \ul{N}(K_\lambda)$). Indeed, the proof of \cite[Lemma 3.2]{esnault-groechenig:rigid} (which proves the claim when $G = \GL_n$) goes through verbatim for $G$, since the only thing that is used about $\GL_n$ is that the union of conjugacy classes of quasi-unipotent elements of index dividing $h$ is closed; and this is true for $G$ also. 

The map $\pi_1^\et(s,\ol{s}) \to \ul{N}(K_\lambda)$ given by $\sigma \mapsto \rho_{i,\lambda}^{\sigma,\top}$ is continuous for the profinite topology on $\pi_1^\et(s,\ol{s})$ and the $\lambda$-adic topology on $\ul{N}(K_{\lambda})$. The $G$-local system $[\rho_{i,\lambda}^\top]$ represented by $\rho_{i,\lambda}^\top$ is isolated, so there is an open subgroup $U$ with $[\rho_{i,\lambda}^{\sigma,\top}] = [\rho_{i,\lambda}^\top]$ for all $\sigma \in U$. Thus replacing $k$ by a finite extension, we may assume that for all $\sigma \in \pi_1^\et(s, \bar{s})$, $\rho_{i,\lambda}^{\sigma,\top}$ is $G(\ol{\QQ}_\lambda)$-conjugate to $\rho_{i,\lambda}^\top$ (we fix an extension of $\lambda$ to $\ol{\QQ}$). By specializing, we see that $\rho_{i,\lambda,\ol{s}}^\sigma = T(\sigma)\rho_{i,\lambda,\ol{s}}T(\sigma)^{-1}$ for some $T(\sigma) \in G(\ol{\QQ}_\lambda)$, where $\rho_{i,\lambda,\ol{s}}^\sigma(\gamma) = \rho_{i,\lambda,\ol{s}}(\sigma\gamma\sigma^{-1})$.

The element $T(\sigma)$ is determined up to multiplication by an element of the centralizer $C_G(\rho_{i,\lambda,\ol{s}})$. Note that $G$-irreducibility of $\rho_{i, \lambda, \bar{s}}$ implies that the geometric monodromy group $G_{i, \lambda}:= \ol{\im(\rho_{i, \lambda, \bar{s}})}$ is reductive, and that its centralizer $C_G(G_{i, \lambda})$ is finite (by \cite[Theorem 4.1]{richardson:conjugacyntuple}, using also that $G$ is semisimple, so has finite center). 
$T(\sigma)$ normalizes $G_{i, \lambda}$, and it follows formally that $\sigma \mapsto T(\sigma)$ defines a homomorphism
\[
\pi_1^\et(s, \bar{s}) \to  N_G(G_{i, \lambda})/C_G(G_{i, \lambda}).
\]
We claim that, after replacing $k$ by a finite extension, we may assume $T(\sigma)$ is an inner automorphism of $G_{i, \lambda}$. We use part of the argument of \cite[Corollaire 1.3.9]{deligne:weil2}. Indeed, any automorphism of $G_{i, \lambda}$ induces, arguing successively, automorphisms of $G_{i, \lambda}^0$, $Z_{G_{i, \lambda}^0}$, and (to abbreviate) $Z^0=Z^0_{G_{i, \lambda}^0}$. Let $r \colon G \to GL_V$ be a faithful finite-dimensional representation, and consider the finite set of characters $X$ of $Z^0$ in $V$. Although $T(\sigma)$ is only defined up to $C_G(G_{i, \lambda})$, conjugating by any representative $T \in N_G(G_{i, \lambda})$ of $T(\sigma)$ shows that as $Z^0$-representations, $r|_{Z^0}$ and $r^T|_{Z^0}$ are isomorphic (as before, the notation means $r^T(z)= r(TzT^{-1})$; this is independent of the choice of representative $T$). Thus $T(\sigma)$ acts by permutations on the finite set $X$, and it follows that a finite power $T(\sigma)^n$ fixes $X$ point-wise. Faithfulness of the representation $r$ implies that these characters generate the character lattice $X^{\bullet}(Z^0)$, and therefore $T(\sigma)^n=T(\sigma^n)$ acts trivially on $Z^0$. The set of outer automorphism of a (not necessarily connected) reductive group that are trivial on the maximal central torus is finite, so enlarging $n$ we have that $T(\sigma^n)$ acts on $G_{i, \lambda}$ as an inner automorphism. Replacing $k$ by its degree $n$ extension, we then conclude that $\sigma \mapsto T(\sigma)$ factors 
\[
\pi_1^{\acute{e}t}(s, \bar{s}) \to G_{i, \lambda}\cdot C_G(G_{i, \lambda})/C_G(G_{i, \lambda}) \xrightarrow{\sim} G_{i, \lambda}/Z_{G_{i, \lambda}}.
\]

We then obtain a ``projective" representation
\[
\mathbb{P}(\rho_{i, \lambda, s}) \colon \pi_1^{\et, p'}(X_s, x_{\bar{s}}) \to G_{i, \lambda}/Z_{G_{i, \lambda}}(\ol{\QQ}_{\lambda})
\]
whose restriction to $\pi_1^{\acute{e}t}(X_{\bar{s}}, \bar{s})$ is $\mathbb{P}(\rho_{i, \lambda, \bar{s}})$ as follows: the rational point $x_s$ induces a semi-direct product decomposition $\pi_1^{\acute{e}t}(X_{s}, s) \cong \pi_1^{\acute{e}t}(X_{\bar{s}}, \bar{s}) \rtimes \pi_1^{\acute{e}t}(s, \bar{s})$, and then we set 
\[
\mathbb{P}(\rho_{i, \lambda, s})(g, \sigma)= \mathbb{P}(\rho_{i, \lambda, \bar{s}})(g) \cdot T(\sigma).
\]

The obstruction to lifting $\mathbb{P}(\rho_{i, \lambda, s})$ to a $G_{i, \lambda}$-valued representation lies in $H^2(\pi_1^{\et, p'}(X_s, x_{\bar{s}}), Z_{G_{i, \lambda}})$, which we analyze using the Hochschild-Serre spectral sequence. The key observation is that the existence of $\rho_{i, \lambda, \bar{s}}$ to begin with implies that the obstruction vanishes after restriction to $\pi_1^{\et, p'}(X_{\bar{s}}, x_{\bar{s}})$. Since $k$ has cohomological dimension 1, vanishing of the full obstruction follows after possibly enlarging $k$ once more to kill a class in $H^1(\pi_1^\et(s, \bar{s}), H^1(\pi_1^{\et, p'}(X_{\bar{s}}, x_{\bar{s}}), Z_{G_{i, \lambda}}))$. Finally, by once more enlarging $k$ we may assume that our lift
\[
\rho_{i, \lambda, s} \colon \pi_1^{\et, p'}(X_s, x_{\bar{s}}) \to G_{i, \lambda}(\ol{\QQ}_{\lambda})
\]
restricts to the given $\rho_{i, \lambda, \bar{s}}$: the two differ by an element of the finite group $H^1(\pi_1^{\et, p'}(X_{\bar{s}}, x_{\bar{s}}), Z_{G_{i, \lambda}})$, and after enlarging $k$ any such homomorphism extends to $\pi_1^{\et, p'}(X_s, x_{\bar{s}})$ (the outer-action of $\pi_1^{\et}(s, \bar{s})$ on this $H^1$ trivializes after a finite restriction, and then we apply the inf-res sequence). By construction, this descended representation has image with Zariski closure equal to $G_{i, \lambda}$.
\end{proof}
Having established the arithmetic descent, we deduce the following important consequence:
\begin{cor}\label{semisimple}
The identity component $G_{i, \lambda}^0$ is semisimple.
\end{cor}
\begin{proof}
The descent $\rho_{i, \lambda, s}$ has (not necessarily connected) reductive monodromy group $G_{i, \lambda}$. It follows from \cite[Corollaire 1.3.9]{deligne:weil2} that $G_{i, \lambda}^0$ is semisimple, since it is also the identity component of the geometric monodromy group. 
\end{proof}

\section{Constructing Companions}\label{companionsection}
\subsection{Application of results of Drinfeld}

We have for each $i$ the representations 
\[
\rho_{i, \lambda, s} \colon \pi_1^\et(X_s, x_{\bar{s}}) \to G(K_{\lambda}) \subset G(\ol{\QQ}_{\lambda}),
\]
whose images have Zariski-closures equal to the semisimple (not necessarily connected) subgroups $G_{i, \lambda} \subset G$ over $\ol{\QQ}_{\lambda}$; recall that by construction $\mr{im}(\rho_{i, \lambda, s})$ has the same Zariski-closure as $\mr{im}(\rho_{i, \lambda, \bar{s}})$. Also, since $i$ ranges over a finite set, we may assume that the same $s$ works for all $i$. For any place $\lambda'$ of $\ol{\QQ}$ not above $p$,
we will construct a $\lambda'$-companion $\rho_{i, \lambda', s}$ with ``the same" algebraic monodromy group as $\rho_{i, \lambda, s}$. Let $(G_i, \phi_{\lambda})$ be a pair consisting of a reductive group $G_i$ over $\ol{\QQ}$ and an isomorphism $G_i \otimes_{\ol{\QQ}} \ol{\QQ}_{\lambda} \xrightarrow{\phi_{\lambda}} G_{i, \lambda}$. We also fix an embedding of groups over $\ol{\QQ}$, $\alpha \colon G_i \to G_{\ol{\QQ}}$, such that $(\alpha \otimes \ol{\QQ}_{\lambda}) \circ \phi_{\lambda}^{-1}$ is $G(\ol{\QQ}_{\lambda})$-conjugate to the given $G_{i, \lambda} \subset G$.

We recall the main theorem of \cite{drinfeld:pross}. For a characteristic zero field $E$, let $\Pross(E)$ denote the groupoid whose objects are pro-semisimple (not necessarily connected) group schemes over $E$, and whose morphisms are group scheme isomorphisms modulo conjugation by the connected component of the identity. There is an extension of scalars equivalence of categories $\Pross(\ol{\QQ}) \xrightarrow{\sim} \Pross(\ol{\QQ}_{\lambda})$ when $\lambda$ is any place of $\ol{\QQ}$ (see \cite[Proposition 2.2.5]{drinfeld:pross} for details). To the smooth variety $X$ over the finite field $k$, Drinfeld associates for each $\lambda$ an object of $\Pross(\ol{\QQ})$ as follows: consider the inverse limit $(\hat{\Pi}_{\lambda}, r_{\lambda})$ over all pairs $(H, r)$ consisting of a semisimple group $H$ over $\ol{\QQ}_{\lambda}$ and a continuous homomorphism with Zariski-dense image $r \colon \pi_1^{\et}(X_s, x_{\bar{s}}) \to H(\ol{\QQ}_{\lambda})$. Via the above equivalence, we descend $\hat{\Pi}_{\lambda}$ to an object $\hat{\Pi}_{(\lambda)}$ of $\Pross(\ol{\QQ})$. If we let $\Pi_{\mr{Fr}}$ denote the set of all Frobenius elements in $\pi_1^{\et}(X_s, x_{\bar{s}})$ as in \cite[\S 1.1.2]{drinfeld:pross}, then the work of L. Lafforgue (\cite[Proposition VII.7]{llafforgue:chtoucas} implies that the universal homomorphism $r_{\lambda}$ induces a diagram of sets
\begin{equation}\label{frobdiagram}
\Pi_{\mr{Fr}} \to (\hat{\Pi}_{(\lambda)} \sslash \hat{\Pi}^0_{(\lambda)})(\ol{\QQ}) \to \pi_1^\et(X_s, x_{\bar{s}}).
\end{equation}
(Note that the GIT quotient, which as in \cite{drinfeld:pross} we will also denote by $[\hat{\Pi}_{(\lambda)}]$, is actually well-defined, even though $\hat{\Pi}_{(\lambda)}$ is only defined up to conjugation by the connected component.) 

The main theorem of \cite{drinfeld:pross} is then:
\begin{thm}[\cite{drinfeld:pross}]\label{drinfeld}
For any two non-archimedean places $\lambda$ and $\lambda'$ of $\ol{\QQ}$ not above $p$, there is a unique isomorphism $\hat{\Pi}_{(\lambda)} \xrightarrow{\sim} \hat{\Pi}_{(\lambda')}$ in $\Pross(\ol{\QQ})$ carrying diagram (\ref{frobdiagram}) for $\lambda$ to the analogous diagram for $\lambda'$.
\end{thm}
We now return to our representations $\rho_{i, \lambda, s}$. By construction and \cite[Proposition 2.3.3]{drinfeld:pross}, $\rho_{i, \lambda, s}$ induces a surjection of (pro-)semisimple groups over $\ol{\QQ}$, $\hat{\Pi}_{(\lambda)} \to G_i$, well-defined up to $G_i^0$-conjugation. 

Composing with the inverse of the isomorphism in Theorem \ref{drinfeld} and extending scalars $\ol{\QQ} \to \ol{\QQ}_{\lambda'}$, we obtain a surjection $\hat{\Pi}_{\lambda'} \to G_{i, \lambda'}$, where we abbreviate $G_{i, \lambda'}= G_i \otimes_{\ol{\QQ}} \ol{\QQ}_{\lambda'}$; this map is also defined up to $G^0_{i, \ol{\QQ}_{\lambda'}}$-conjugation. This furnishes us with the desired companion, namely the composite
 \begin{center}
    \begin{tikzcd}
\pi_1^{\et}(X_s, x_{\bar{s}}) \arrow[r,"{r_{\lambda'}}"] \arrow[rr, bend left, "\rho_{i, \lambda', s}"] & \hat{\Pi}_{\lambda'}(\ol{\QQ}_{\lambda'}) \arrow[r] & G_{i, \lambda'}(\ol{\QQ}_{\lambda'}).
    \end{tikzcd}
\end{center}
It is evident from the construction that the Zariski-closure $\ol{\im(\rho_{i, \lambda', s})}$ is equal to $G_{i, \lambda'}$, and that $\rho_{i, \lambda, s}$ and $\rho_{i, \lambda', s}$ are everywhere locally compatible (in the sense made precise by diagram (\ref{frobdiagram})).
\begin{prop}\label{companion properties}
 As $i=1, \ldots, \mathfrak{N}$ varies, the representations $\rho_{i, \lambda', \bar{s}} \colon \pi_1^\et(X_{\bar{s}}, x_{\bar{s}}) \to G_{i, \lambda'}(\ol{\QQ}_{\lambda'})$
 \begin{enumerate}
     \item are tamely ramified;
     \item are $G$-cohomologically rigid;
     \item are pair-wise distinct as $G(\ol{\QQ}_{\lambda'})$-conjugacy classes; 
     and
     \item have quasi-unipotent local monodromies with index dividing $h$ (in the sense of Remark \ref{quasi-unipotent monodromy}). 
 \end{enumerate}
\end{prop}
\begin{proof}

(1): By \cite[Proposition 4.2]{kerz-schmidt:tameness}, it is enough to check all restrictions of $\rho_{i,\lambda', s}$ to regular curves mapping to ${X}_s$, and with image not contained in $D$, are tamely ramified. 
Let ${Z} \to {X}_s$ be one such curve and fix a geometric base point $z \in Z$ and path from (the image of) $z$ to $x_{\bar{s}}$ in ${X}_s$. Choose a faithful representation $R \colon G_i \to \GL_n$ over $\ol{\QQ}$ and write $R_{i,\lambda,s}$ for the composition
    \[ \pi_1^{\et}(Z, z) \to \pi_1^{\et}(X_s, x_{\bar{s}}) \xrightarrow{\rho_{i,\lambda, s}} G(\ol{\QQ}_\lambda) \xrightarrow{R} \GL_n(\ol{\QQ}_\lambda), \]
and similarly for \ $R_{i,\lambda',s}$. These are, by the compatibility of $\rho_{i, \lambda, s}$ and $\rho_{i, \lambda', s}$, compatible representations.
Considering them as representations of the Galois group of the function field $k(Z)$ of $Z$, \cite[Th\'{e}or\`{e}me 9.8]{deligne:constantes} shows that the local representations of $R_{i,\lambda,s}$ and $R_{i,\lambda',s}$ at each place of $k(Z)$ have isomorphic semisimplified Weil-Deligne representations; regarding a Weil-Deligne representation as a pair consisting of a Weil group representation $r$ (smooth on the inertia subgroup) and a nilpotent matrix $N$, this semisimplification remembers the restriction of $r$ to the inertia subgroup, and in particular it detects whether or not the representation is tamely ramified. As $\rho_{i,\lambda, s}$ factors through the prime-to-$p$ quotient, it is tame, so \cite[Proposition 4.2]{kerz-schmidt:tameness} (which is a corollary of Abhyankar's lemma \cite[XIII Proposition 5.2]{sga1}) shows that $R_{i,\lambda,s}$ is tamely ramified. We conclude from \cite{deligne:constantes} that $R_{i,\lambda',s}$ is also tamely ramified. Since this holds for all such $Z \to X_s$, we conclude again by \cite[Proposition 4.2]{kerz-schmidt:tameness} that $\rho_{i, \lambda', s}$ (equivalently, $R \circ \rho_{i, \lambda', s}$) is tamely ramified.

(4): 
As explained in Remark \ref{Spreading out the curve}, we have a smooth curve $\ol{C} \subset \ol{X}$ that is a complete intersection with smooth complement; $\ol{C}$ does not meet the singular locus of $D$ and intersects each $D_j$ transversely. We also have points $x \in \ol{X} \setminus D$ and $y_j \in D_j \cap \ol{C}$. 
Recall that $S$ is chosen as in Remark \ref{Spreading out the curve} so that $(\ol{C}, y_j)$ (along with these good properties) spreads out to $(\ol{C}_S, y_{j,S})$ over $S$, and $(\ol{C}_s, y_{j,s})$ is the fiber over $s$. As in (1), choose a faithful representation $R \colon G_i \to \GL_n$ over $\ol{\QQ}$. Write $R_{i,\lambda,s}$ for the $\lambda$-adic local system on the curve $C_s$ given by restricting $\rho_{i,\lambda, s}$ to $C_s$ and then composing with $R$ (and similarly for $R_{i,\lambda',s}$). The index of quasi-unipotence around $D_j$ of a representation of $\pi_1^{\et,t}(X_s, x_{\bar{s}})$ is equal to the index around $y_j$ of its restriction to $\pi_1^{\et,t}(C_s, x_{\bar{s}})$: indeed, locally at $y_j$ the divisor $D_j$ is the vanishing locus of a function $f_j$ that restricts (via $\mc{O}_{\ol{X}_{\bar{s}}, y_j} \to \mc{O}_{\ol{C}_{\bar{s}, y_j}}$) to a uniformizer of $\ol{C}_{\bar{s}}$ at $y_j$ (since the intersection is transverse). The claim then follows by Abhyankar's Lemma (\cite[XIII Proposition 5.2]{sga1}) upon noting that a Kummer cover ramified along $D_j$ of degree $n$ pulls back to a Kummer cover ramified along $y_j$ of degree $n$. 

As noted in \ref{quasi-unipotent monodromy}, the index of quasi-unipotence of $\rho_{i, \lambda, s}$ around $D_j$ divides $h$, so the same holds for $R_{i, \lambda, s}$ around $y_j$. We again invoke Deligne's theorem \cite[Th\'{e}or\`{e}me 9.8]{deligne:constantes} and find that $R_{i, \lambda', s}$ has the same index of quasi-unipotence around $y_j$, dividing $h$. Finally, the same observation just used implies that $\rho_{i, \lambda', s}$ has index of quasi-unipotence along $D_j$ also dividing $h$.



(2): Let $\mc{A}_{i,\lambda, s}$ be the local system on $X_s$ corresponding to the representation 
    \[ \pi_1^\et(X_s, x_{\ol{s}}) \xrightarrow{\rho_{i,\lambda, s}} G(\ol{\QQ}_\lambda) \to \GL(\fg^{\mr{der}}(\ol{\QQ}_\lambda)) \]
and likewise for $\mc{A}_{i,\lambda',s}$. The determinant of each irreducible component of $\mc{A}_{i,\lambda,s}$ is finite order since $G_{i, \lambda}$ is semisimple, 
and hence $\mc{A}_{i,\lambda,s}$ is pure of weight zero by Lafforgue \cite{llafforgue:chtoucas}. The sheaves $\mc{A}_{i,\lambda,s}$ and $\mc{A}_{i,\lambda',s}$ are compatible by the remark above. 

We argue as in the proof of Theorem 1.1 of \cite{esnault-groechenig:rigid}. Compatibility gives an  equality of $L$-functions
    \[ L(X_{\ol{s}}, \mc{A}_{i,\lambda,s}) = L(X_{\ol{s}}, \mc{A}_{i,\lambda',s}) \]
and since $\mc{A}_{i,\lambda,s}$ is tame and weight $0$, \cite[Lemma 3.4]{esnault-groechenig:rigid} shows that
\[ h^1(\ol{X}_{\ol{s}}, j_{!\ast}\mc{A}_{i,\lambda',\ol{s}}) = h^1(\ol{X}_{\ol{s}}, j_{!\ast}\mc{A}_{i,\lambda, \ol{s}}). \]
The latter is $0$, as a consequence of local acyclicity and the fact that $\rho_{i,\lambda}^\top$ is cohomologically rigid. Thus $\rho_{i,\lambda',\ol{s}}$ is cohomologically rigid as required.

(3): If $\rho_{i, \lambda', s}$ were $G(\ol{\QQ}_{\lambda'})$-conjugate to $\rho_{j, \lambda', s}$, then again invoking \cite[Proposition 2.3.3]{drinfeld:pross} we find that these maps are induced by homomorphisms $\hat{\Pi}_{(\lambda')} \to G_{\ol{\QQ}}$ that are $G(\ol{\QQ})$-conjugate, and hence that $\rho_{i, \lambda, s}$ and $\rho_{j, \lambda, s}$ are $G(\ol{\QQ}_{\lambda})$-conjugate (note that $G$ is connected). 
Moreover, the same argument works if we replace $k(s)$ by any finite extension. We claim the same for the representations $\rho_{i, \lambda', \bar{s}}$. We will first check that $\rho_{i, \lambda', s}$ and $\rho_{i, \lambda', \bar{s}}$ have the same Zariski-closure (namely $G_{i, \lambda'}$), by the corresponding property for the $\lambda$-companions. Write $G_{i, \lambda', \bar{s}}$ for the Zariski-closure of the image of $\rho_{i, \lambda', \bar{s}}$ (for any $\lambda'$, including $\lambda$). Since the commutator subgroup of $G_{i, \lambda'}$ is contained in $G_{i, \lambda', \bar{s}}$, and $G_{i, \lambda'}^0$ is semisimple, $G_{i, \lambda'}^{0, \mr{der}}=G_{i, \lambda'}^0$ is equal to $G_{i, \lambda', \bar{s}}^0$. Thus $G_{i, \lambda'}/G_{i, \lambda', \bar{s}}$ is isomorphic to the quotient of component groups $\pi_0(G_{i, \lambda'})/\pi_0(G_{i, \lambda', \bar{s}})$. If this were non-trivial, it would be generated by the image of the geometric Frobenius $F_s \in \pi_1^\et(s, \bar{s}) \xrightarrow{x_s} \pi_1^\et(X_s, x_{\bar{s}})$, and, for some integer $n$, $F_s^n$ would have image in $\pi_0(G_{i, \lambda', \bar{s}})$. Thus after replacing $k(s)$ by a finite extension, we may assume $G_{i, \lambda', \bar{s}}=G_{i, \lambda'}$, as desired.

Now assume that $\rho_{i, \lambda', \bar{s}}$ and $\rho_{j, \lambda', \bar{s}}$ are conjugate by an element of $G(\ol{\QQ}_{\lambda'})$. Replacing $\rho_{i, \lambda', \bar{s}}$ by such a conjugate, we may assume the two homomorphisms $\rho_{i, \lambda', s}$ and $\rho_{j, \lambda', s}$ are literally equal on $\pi_1^{\et}(X_{\bar{s}}, x_{\bar{s}})$. Let $F_s \in \pi_1^{\et}(X_s, x_{\bar{s}})$ be the image of the generating geometric Frobenius element of $\pi_1^\et(s, \bar{s})$. Then for all $\gamma \in \pi_1^{\et}(X_{\bar{s}}, x_{\bar{s}})$, $\rho_{i, \lambda', \bar{s}}(F_s \gamma F_s^{-1})=\rho_{j, \lambda', \bar{s}}(F_s \gamma F_s^{-1})$, hence $\rho_{j, \lambda', s}(F_s)^{-1} \rho_{i, \lambda', s}(F_s)$ belongs to the center $Z_{G_{j, \lambda'}}$ (we have used the above assertion about equality of arithmetic and geometric monodromy groups). As this center is finite, replacing $k(s)$ by a finite extension we have $\rho_{i, \lambda', s}=\rho_{j, \lambda', s}$, a contradiction: indeed, then their $\lambda$-companions, and consequently $\rho_{i, \lambda, \bar{s}}$ and $\rho_{j, \lambda, \bar{s}}$, would then be equivalent. 
\end{proof}

\subsection{Deduction of the main theorem}
\begin{proof}[Proof of Theorem \ref{mainthm}]
Now, since the companions $\rho_{i, \lambda', \bar{s}}$ are tamely ramified, we can pull them back along the tame specialization map (\cite[Corollary A.12]{lieblich-olsson-pi1}) $\pi_1^{\et, t}(X, x) \twoheadrightarrow \pi_1^{\et, t}(X_{\bar{s}}, x_{\bar{s}})$ to obtain complex local systems
\[
\rho^{\top}_{i, \lambda'} \colon \pi_1^{\top}(X, x) \to G_{i, \lambda'}(\ol{\QQ}_{\lambda'}) \subset G(\ol{\QQ}_{\lambda'})
\]
with image in fact contained in $G(\ol{\ZZ}_{\lambda'})$ (because they are constructed from \'etale local systems). The Zariski closure of $\mr{im}(\rho_{i,\lambda'}^\top)$ is $G_{i,\lambda'}$ since this is true for $\mr{im}(\rho_{i,\lambda',\ol{s}})$ and the tame specilization map is surjective. Thus $\rho_{i,\lambda'}^\top$ has (not necessarily connected) semisimple monodromy. By the Betti-\'etale comparison isomorphism, the $\rho^{\top}_{i, \lambda'}$ are still cohomologically rigid (now as local systems on $X$), and they are still inequivalent and $G$-irreducible with quasi-unipotent local monodromy of index dividing $h$. If we apply a field isomorphism $\iota \colon \CC \xrightarrow[/\ol{\QQ}]{\sim} \ol{\QQ}_{\lambda'}$ to the elements of the finite set $\mc{S}(G, h)$, we obtain a set of $\mathfrak{N}$ distinct $G(\ol{\QQ}_{\lambda'})$-local systems on $X$ that are $G$-irreducible, $G$-cohomologically rigid, and with monodromy at infinity quasi-unipotent of index dividing $h$, and these are by construction a complete set of such local systems. Our $\rho^\top_{i, \lambda'}$, for $i=1, \ldots, \mathfrak{N}$, are thus up to $G(\ol{\QQ}_{\lambda'})$-isomorphism a full collection of representatives of this set $\iota \mc{S}(G, h)$. Returning to our original $\rho \colon \pi_1^\top(X, x) \to G(\mc{O}_{K, \Sigma})$, we see that for any $\lambda'$ above a place in $\Sigma$, the composite 
\[
\pi_1^\top(X, x) \xrightarrow{\rho} G(\mc{O}_{K, \Sigma}) \subset G(\ol{\QQ}_{\lambda'}) 
\]
is also a member of $\iota \mc{S}(G, h)$. We conclude that $\rho$ can be conjugated by an element of $G(\ol{\QQ}_{\lambda'})$ into $G(\ol{\ZZ}_{\lambda'})$, since this holds for each $\rho_{i, \lambda'}^\top$. Combining this argument for all places in $\Sigma$ with our integrality criterion, Lemma \ref{integralityprop}, we deduce that for some finite extension $L/K$, $\rho$ is $G(L)$-conjugate to a homomorphism $\pi_1^{\top}(X, x) \to G(\mc{O}_L)$, concluding the proof.
\end{proof}

\bibliographystyle{amsalpha}
\bibliography{biblio.bib}
\end{document}